\documentclass{llncs}

\usepackage[utf8]{inputenc}
\usepackage{enumerate}
\usepackage[shortlabels]{enumitem}
\usepackage{graphicx}
\usepackage{epstopdf}
\usepackage{amsfonts}
\usepackage{amssymb}
\usepackage{amsmath}
\usepackage[labelfont=bf]{caption}
\usepackage{sectsty}
\usepackage{wasysym}
\usepackage{hyperref}

\spnewtheorem*{proposition*}{Proposition}{}{\itshape}
\spnewtheorem*{theorem*}{Theorem}{}{\itshape}
\spnewtheorem*{lemma*}{Lemma}{}{\itshape}

\begin{document}

\title{Characterizing subclasses of cover-incomparability graphs by forbidden subposets\thanks{
The final version of this preprint was published in \emph{Order}. The paper is available on
\url{https://doi.org/10.1007/s11083-018-9470-7}
}
}

\author{Jan Bok \inst{1} \and Jana Maxov\'{a} \inst{2}}

\institute{Computer Science Institute of Charles University,
Malostransk\'{e} n\'{a}m. 25, 118~00~~Praha~1, Czech Republic. \\
\email{bok@iuuk.mff.cuni.cz}
\and
Department of Mathematics, Institute of Chemistry and Technology, Prague,
Technick\'{a} 5, 166~28~~Praha~6, Czech Republic. \\
\email{maxovaj@vscht.cz}
}

\date{}

\maketitle

\begin{abstract}
In this paper we continue investigations of cover-incomparability graphs of
finite partially ordered sets (see \cite{Bres,Bres2,Bres3,Bres4} and \cite{Max,MaxDH}). We consider in some
detail the distinction between cover-preserving subsets and isometric subsets
of a partially ordered set. This is critical to understanding why forbidden
subposet characterizations of certain classes of cover-incomparability graphs
in \cite{Bres} and \cite{Bres3} are not valid as presented. Here we provide examples,
investigate the root of the difficulties, and formulate and prove valid
revisions of these characterizations.
\end{abstract}

\section{Introduction}
\label{s:intro}

In this paper we deal with posets and graphs associated to them. There are
several ways how to associate a graph $G$ to a given poset $P$. The vertex set
$V(G)$ is usually the set of points of $P$. Depending on the edge-set $E(G)$,
we may obtain among others the {\it comparability graph} of P ($x$ and $y$ are
adjacent iff $x < y$ or $y < x$), the {\it incomparability graph} of P ($x$
and $y$ are adjacent iff $x$ and $y$ are incomparable), the {\it cover graph}
of P ($x$ and $y$ are adjacent iff $x$ covers $y$ or vice versa) or the {\it
cover-incomparability graph} of P ($x$ and $y$ are adjacent iff $x$ covers
$y$, or $y$ covers $x$, or $x$ and $y$ are incomparable). The incomparability
graph of $P$ is of course just the complement of its comparability graph,
while the cover-incomparability graph of $P$ is the union of the cover graph
and the incomparability graph of $G$.

Cover graphs, comparability graphs and incomparability graphs are standard
ways how  to associate a graph to a given poset, while the
notion of cover-incom\-pa\-ra\-bility graph is new. It was introduced in \cite{Bres}. This notion
was motivated by the theory of transit functions on posets. It turns out that
the underlying graph $G_{P}$ of the standard transit function $T_P$ on the
poset $P$ is exactly the cover-incomparability graph of $P$ (see \cite{Bres}
for details).

Cover-incomparability graphs have been so far approached in two different ways.
One possibility is to try to characterize graphs that are cover-incomparability graphs. In \cite{Max} it was proved that the recognition
problem for cover-incomparability graphs is in general NP-complete. On the
other hand there are classes of graphs (such as trees, Ptolemaic graphs,
distance-hereditary graphs, block graphs, split graphs or $k$-trees) for which
the recognition problem can be solved in linear time
(see \cite{Bres2,Bres3,MaxDH,Maxktrees} for details and proofs). 

Another approach is to study posets whose cover-incomparability graphs have
certain property. Posets whose cover-incomparability graphs are chordal,
Ptolemaic, distance-hereditary, claw-free or cographs were characterized in
\cite{Bres} and \cite{Bres4}. Unfortunately, there is a mistake that
originated in \cite{Bres} and continued in \cite{Bres4} and several statements
from these papers do not hold as they are stated. In this paper we correct the
mistake and reformulate the corresponding statements so that they hold.

Our paper is organized as follows.
In Section \ref{s:terminology} we give an overview of terminology and basic
properties of cover-incomparability graphs. In Section~\ref{s:counterexamples}
we present counterexamples to Theorem~4.1 from~\cite{Bres3}, Lemma~4.4 and 4.5
from~\cite{Bres} and to Proposition~5.1 from~\cite{Bres}. In
Section~\ref{s:oprava} we show that the mistake originated in Theorem~2.4 of
\cite{Bres}. We reformulate this statement and give a corrected proof of it.
In addition, we reformulate all the above mentioned statements so that they
hold.

\section{Terminology and basic properties}
\label{s:terminology}

Let $P=(V, \leq)$ be a poset. We will use the following notation. For $u,v \in V$ we write:
\begin{itemize}
\item $u<v$ if $u\leq v$ and $u \neq v$.
\item $u\prec v$ if $u<v$ and there is no $z \in V$ such that $u<z<v$. We say that $v$ {\it covers} $u$.
\item $u\prec\prec v$ if       $u<v$ and  $\neg (u \prec v)$.
\item $u \parallel v$ if $u$ and $v$ are incomparable.
\end{itemize}

\begin{definition}
  For a given poset $P=(V, \leq)$, let $G_P=(V,E)$ be a graph with $E=\{ \{u,v\} \ |\ u\prec v \ {\rm or}\ v\prec u \ {\rm or} \ u\parallel v \}$. Then we say that $G_P$ is the \emph{cover-incomparability graph} of $P$ (or the C-I graph of $P$ for short).
\end{definition}

Note that for any $u,v \in V(G_P)$, $u \neq v$ we have  $\{u,v\}
\notin E(G_P) \Leftrightarrow u\prec\prec v \ {\rm or} \
v\prec\prec u.$

As this is crucial for the rest of our paper let us define properly the following three concepts.

\begin{definition}
 Let  $P=(V_P, \leq_P)$ be a poset.
   \begin{itemize}
   \renewcommand{\labelitemi}{$\bullet$}
   \item We say that $Q=(V_Q, \leq_Q)$ is a \emph{subposet} of $P=(V_P, \leq_P)$ if
   \begin{enumerate}
   \item $V_Q \subseteq V_P$ and
   \item for any $u,v \in V_Q$ we have $u \leq_Q v \Leftrightarrow u \leq_P v$.
   \end{enumerate}
\vspace{0.1cm}
\item We say that $R=(V_R, \leq _R)$ is an \emph{isometric subposet} of $P=(V_P, \leq_P)$ if
   \begin{enumerate}
 \item $V_R \subseteq V_P$ and
  \item for any $u,v \in V_R$ we have $u \leq_R v \Leftrightarrow u \leq_P v$ and
   \item for any $u,v \in V_R$ such that $u \leq_R v$ there exists a chain of a shortest length between $u$ and $v$ in $P$ is also in $R$.
   \end{enumerate}
\vspace{0.1cm}
   \item We say that $S=(V_S, \leq_S)$ is a \emph{$\prec$-preserving subposet} of $P=(V_P, \leq_P)$ if
   \begin{enumerate}
    \item $V_S \subseteq V_P$ and
  \item for any $u,v \in V_S$ we have $u \leq_S v \Leftrightarrow u \leq_P v$ and
   \item for any $u,v \in V_S$ we have $u \prec_S v \Leftrightarrow u \prec_P v$.
   \end{enumerate}
   \end{itemize}
\end{definition}

Note that an isometric subposet is always $\prec$-preserving but there are
$\prec$-preserving subposets that are not isometric. For example, the poset $P'$
depicted in Fig.~\ref{f:nonisometric} is a nonisometric $\prec$-preserving
subposet of $P$ in Fig.~\ref{f:nonisometric}.

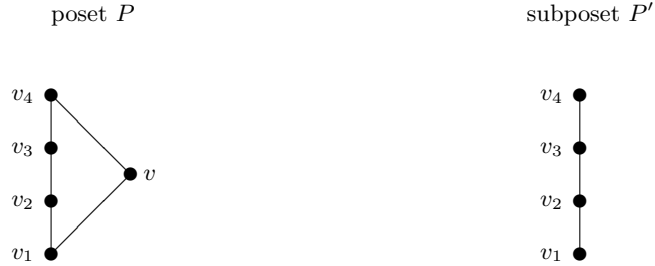
\begin{figure}[h]
\begin{center}
\begin{picture}(300,120)
\put(50,10){\circle*{5}}
\put(50,30){\circle*{5}}
\put(50,50){\circle*{5}}
\put(50,70){\circle*{5}}
\put(80,40){\circle*{5}}
\put(50,10){\line(0,1) {60}}
\put(50,10){\line(1,1) {30}}
\put(50,70){\line(1,-1) {30}}
\put(35,8) {$v_1$}
\put(35,28) {$v_2$}
\put(35,48) {$v_3$}
\put(35,68) {$v_4$}
\put(85,38) {$v$}
\put(50,98) {poset $P$}

\put(230,98) {subposet $P'$}
\put(250,10){\circle*{5}}
\put(250,30){\circle*{5}}
\put(250,50){\circle*{5}}
\put(250,70){\circle*{5}}
\put(250,10){\line(0,1) {60}}
\put(235,8) {$v_1$}
\put(235,28) {$v_2$}
\put(235,48) {$v_3$}
\put(235,68) {$v_4$}
\end{picture}
\caption{A nonisometric $\prec$-preserving subposet.}
\label{f:nonisometric}
\end{center}
\end{figure}

Let us also mention a few easy observations about C-I graphs. They follow
immediately from the definition.

\begin{lemma}
\label{l:basic} Let $P=(V, \leq)$ be a poset and $G_P=(V,E)$ its
C-I graph. Then the following holds.
\begin{enumerate}[(i)]
\item $G_P$ is connected.
\item If $U \subseteq V$ is an antichain in $P$ then $U$ induces a complete subgraph in $G_P$.
\item If $I \subseteq V$ is an independent set in $G_P$ then all points of $I$ lie on a common chain in $P$.
\item There are at most $2$ vertices of degree $1$ in $G_P$.
\item If $P^* =(V, \leq^*)$ is the dual poset to $P$ (i.e. $u \leq v$ in P
$\Leftrightarrow v \leq^*u$ in $P^*$), then $G(P^*) = G_P$.
\item If the vertices $x,y,z$ form a triangle in $G_P$ then at least two of them are incomparable.
\item Let $x,y,z$ be vertices of $G_P$ such that $xy \in E$,  $xz \notin E$, $yz \notin E$.
Then $(x \prec\prec \  z \ {\rm and} \ y \prec\prec \  z)$ or $(z \prec\prec \  x \ {\rm and}\  z \prec\prec \  y)$.
\end{enumerate}
\end{lemma}

\section{Counterexamples} \label{s:counterexamples}

In this section we present counterexamples to several statements from \cite{Bres} and \cite{Bres3}. Let us start with the easiest case, with Proposition 5.1 from \cite{Bres}.

{\bf\subsection{A counterexample to Proposition 5.1 from \cite{Bres} }}
First we cite the statement of this proposition in the original text:

\begin{proposition*}[Proposition~5.1~\cite{Bres}]
Let $P$ be a poset. Then $G_P$ contains an induced claw if and only if $P$ contains one of $S_1$, $S_2$ or $S_3$ as an isometric subposet, see Fig.~\ref{f:posetyS}.
\end{proposition*}

\begin{figure}
\begin{center}
\includegraphics[width=0.85\textwidth]{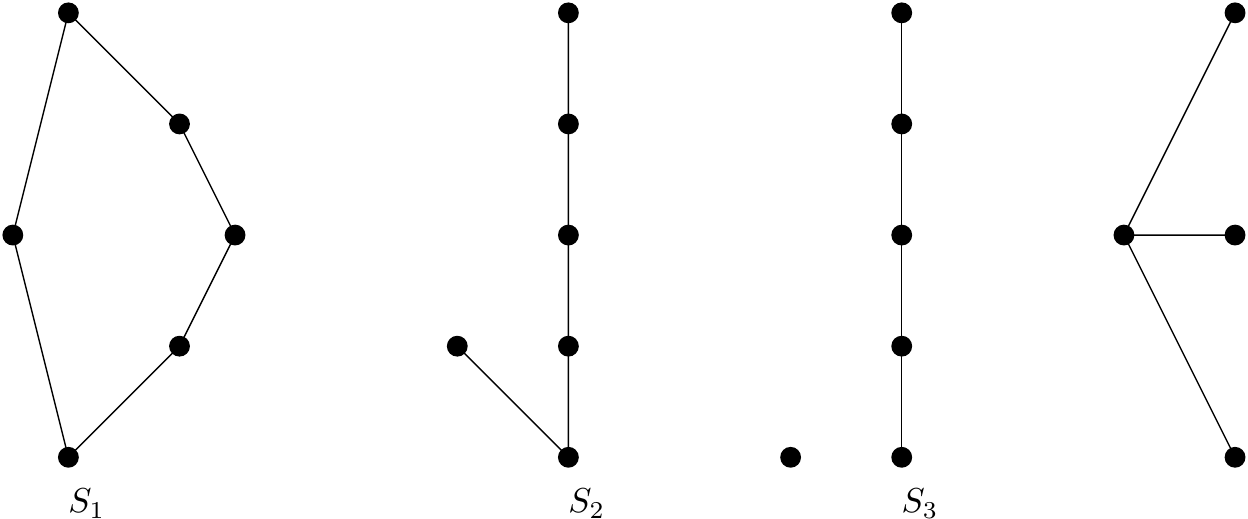}
\caption{Subposets $S_1$, $S_2$ and $S_3$ and the claw.}
\label{f:posetyS}
\end{center}
\end{figure}

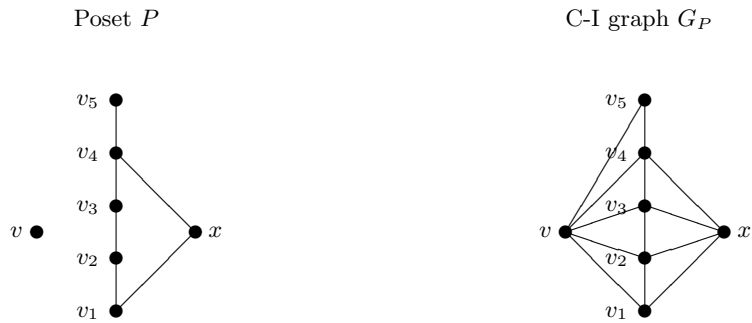
\begin{figure}
\begin{center}
\begin{picture}(300,120)
\put(50,10){\circle*{5}}
\put(50,30){\circle*{5}}
\put(50,50){\circle*{5}}
\put(50,70){\circle*{5}}
\put(50,90){\circle*{5}}
\put(20,40){\circle*{5}}
\put(80,40){\circle*{5}}
\put(50,10){\line(0,1) {80}}
\put(50,10){\line(1,1) {30}}
\put(50,70){\line(1,-1) {30}}
\put(35,8) {$v_1$}
\put(35,28) {$v_2$}
\put(35,48) {$v_3$}
\put(35,68) {$v_4$}
\put(35,88) {$v_5$}
\put(85,38) {$x$}
\put(10,38) {$v$}
\put(34,118) {Poset $P$}

\put(250,10){\circle*{5}}
\put(250,30){\circle*{5}}
\put(250,50){\circle*{5}}
\put(250,70){\circle*{5}}
\put(250,90){\circle*{5}}
\put(220,40){\circle*{5}}
\put(280,40){\circle*{5}}
\put(250,10){\line(0,1) {80}}
\put(250,10){\line(1,1) {30}}
\put(250,30){\line(3,1) {30}}
\put(250,50){\line(3,-1) {30}}
\put(250,70){\line(1,-1) {30}}
\put(220,40){\line(1,1) {30}}
\put(220,40){\line(1,-1) {30}}
\put(220,40){\line(3,1) {30}}
\put(220,40){\line(3,-1) {30}}
\put(220,40){\line(3,5) {30}}

\put(235,8) {$v_1$}
\put(235,28) {$v_2$}
\put(235,48) {$v_3$}
\put(235,68) {$v_4$}
\put(235,88) {$v_5$}
\put(285,38) {$x$}
\put(210,38) {$v$}
\put(220,118) {C-I graph $G_P$}
\end{picture}
\caption{A counterexample to Proposition 5.1.}
\label{f:protipriklad1}
\end{center}
\end{figure}

This statement does not hold. Let $P$ be the poset depicted in
Fig.~\ref{f:protipriklad1}. Clearly, neither $S_1$ nor $S_2$ are subposets of
$P$. $S_3$ is a subposet of $P$ but it is \emph{not an isometric subposet} of
$P$. This is because there is a chain of length two between $u$ and $v$ in $P$
while there is no chain of length two between $u$ and $v$ in $S_3$. Thus $P$
does not contain any of $S_1$, $S_2$ and $S_3$ as an isometric subposet. But
$G_P$ contains an induced claw on vertices $v,v_1,v_3,v_5$, a contradiction.

Counterexamples to other statements can be derived in a similar way:

\subsection{A counterexample to Lemma~4.4 from \cite{Bres}}

\begin{proposition*}[Proposition 4.4 in \cite{Bres}]
Let $P$ be a poset. Then $G_P$ contains an induced house if and only if $P$ contains one of $R_1$, $R_2$, $R_3$, $R_4$ or $R_5$ as an isometric subposet, see Figure~\ref{f:posetyR}.
\end{proposition*}

\begin{figure}
\begin{center}
\includegraphics[width=0.90\textwidth]{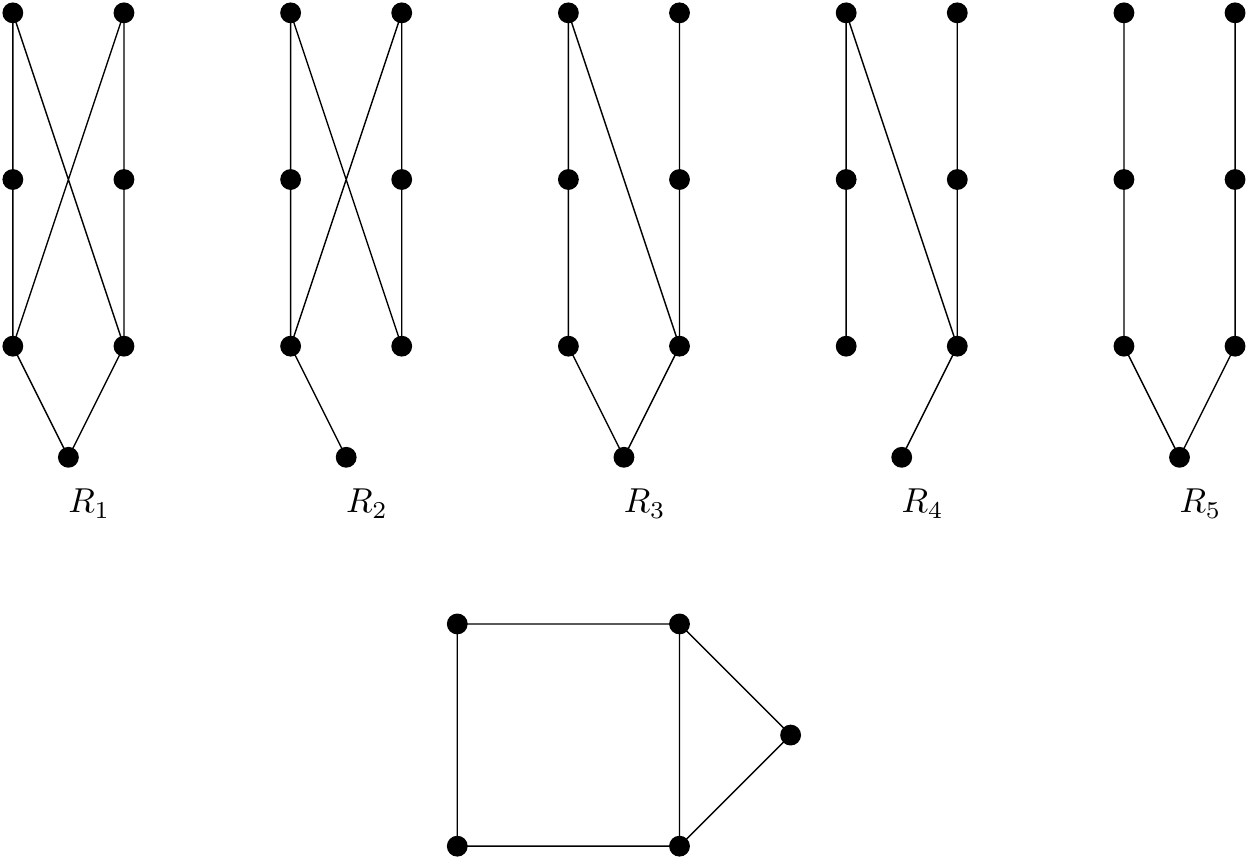}
\caption{Subposets $R_i$, $i=1, \dots 5$ and the house.}
\label{f:posetyR}
\end{center}
\end{figure}

\begin{figure}[h]
\begin{center}
\begin{picture}(450,90)
\put(50,10){\circle*{5}}
\put(30,30){\circle*{5}}
\put(70,30){\circle*{5}}
\put(30,50){\circle*{5}}
\put(70,50){\circle*{5}}
\put(30,70){\circle*{5}}
\put(70,70){\circle*{5}}
\put(30,30){\circle*{5}}
\put(100,40){\circle*{5}}
\put(100,40){\line(-1,1) {30}}
\put(100,40){\line(-5,-3) {50}}
\put(70,30){\line(0,1) {40}}
\put(30,30){\line(0,1) {40}}
\put(50,10){\line(1,1) {20}}
\put(50,10){\line(-1,1) {20}}

\put(35,8) {$v_1$}
\put(15,28) {$v_2$}
\put(15,48) {$v_3$}
\put(15,68) {$v_4$}
\put(56,28) {$v_5$}
\put(56,48) {$v_6$}
\put(56,68) {$v_7$}
\put(105,38) {$x$}

\put(45,90) {Poset $P$}
\put(210,90) {Induced house in $G_P$}

\put(300,10){\circle*{5}}
\put(280,30){\circle*{5}}
\put(320,30){\circle*{5}}
\put(280,70){\circle*{5}}
\put(320,70){\circle*{5}}

\put(320,30){\line(0,1) {40}}
\put(280,30){\line(0,1) {40}}
\put(300,10){\line(1,1) {20}}
\put(300,10){\line(-1,1) {20}}
\put(280,30){\line(1,0) {40}}
\put(280,70){\line(1,0) {40}}

\put(285,8) {$v_1$}
\put(265,28) {$v_2$}
\put(265,68) {$v_7$}
\put(326,28) {$v_5$}
\put(326,68) {$v_4$}

\put(200,10){\circle*{5}}
\put(180,30){\circle*{5}}
\put(220,30){\circle*{5}}
\put(180,70){\circle*{5}}
\put(220,70){\circle*{5}}

\put(250,45){$=$}

\put(200,10){\line(-1,1) {20}}
\put(200,10){\line(1,1) {20}}
\put(180,30){\line(1,1) {40}}
\put(180,70){\line(1,-1) {40}}
\put(180,30){\line(1,0) {40}}
\put(180,70){\line(1,0) {40}}

\put(185,8) {$v_1$}
\put(165,28) {$v_2$}
\put(165,68) {$v_4$}
\put(226,28) {$v_5$}
\put(226,68) {$v_7$}
\end{picture}
\end{center}
\caption{A counterexample to Lemma 4.4}
\label{f:protipriklad2}
\end{figure}
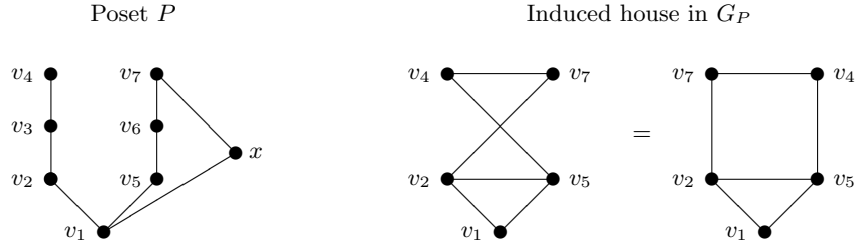

Let $P$ be the poset depicted in Fig.~\ref{f:protipriklad2}. It is easy to see
that it is a counterexample to Lemma 4.4 \cite{Bres}. Indeed, $P$ does not
contain any of the posets $R_1$, $R_2$, $R_3$, $R_4$ or $R_5$ as an {\bf
isometric subposet}. But $G_P$ contains an induced house on vertices $v_1,v_2,
v_4,v_5, v_7$, a contradiction.

{\bf\subsection{A counterexample to Lemma 4.5 from \cite{Bres} }}

\begin{proposition*}[Proposition 4.5 in \cite{Bres}]
Let $P$ be a poset. Then $G_P$ contains an induced domino if and only if $P$ contains one of $D_1$, $D_2$, $D_3$, $D_4$, $D_5$, $D_6$ or $D_7$ as an isometric subposet, see Fig.~\ref{f:posetyD}.
\end{proposition*}

\begin{figure}[h]
\begin{center}
\includegraphics[width=0.80\textwidth]{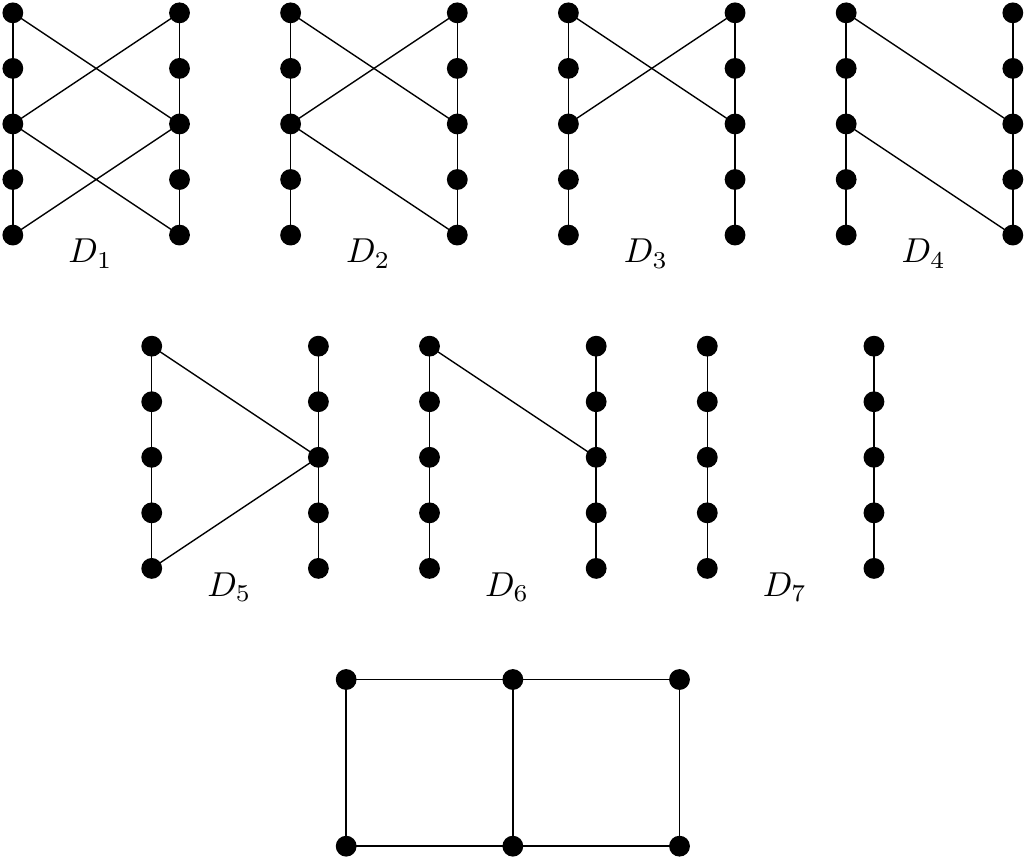}
\caption{Subposets $D_i$, $i=1, \dots 7$, and the domino graph.}
\label{f:posetyD}
\end{center}
\end{figure}

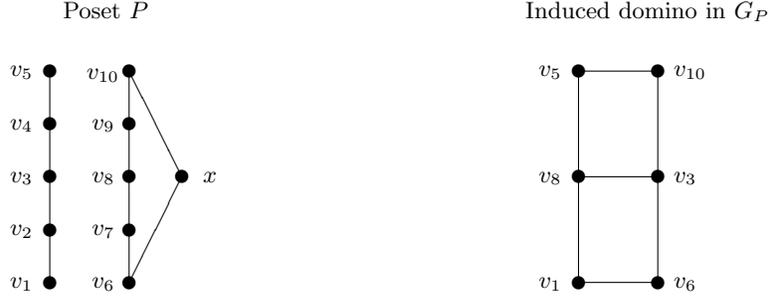
\begin{figure}[h]
\begin{center}
\vspace{0.3cm}
\begin{picture}(300,100)
\put(50,10){\circle*{5}}
\put(50,30){\circle*{5}}
\put(50,50){\circle*{5}}
\put(50,70){\circle*{5}}
\put(50,90){\circle*{5}}

\put(20,10){\circle*{5}}
\put(20,30){\circle*{5}}
\put(20,50){\circle*{5}}
\put(20,70){\circle*{5}}
\put(20,90){\circle*{5}}

\put(70,50){\circle*{5}}
\put(70,50){\line(-1,2) {20}}
\put(70,50){\line(-1,-2) {20}}

\put(50,10){\line(0,1) {80}}
\put(20,10){\line(0,1) {80}}

\put(5,8) {$v_1$}
\put(5,28) {$v_2$}
\put(5,48) {$v_3$}
\put(5,68) {$v_4$}
\put(5,88) {$v_5$}

\put(36,8) {$v_6$}
\put(36,28) {$v_7$}
\put(36,48) {$v_8$}
\put(36,68) {$v_9$}
\put(34,87) {$v_{10}$}

\put(78,48) {$x$}

\put(250,10){\circle*{5}}
\put(250,50){\circle*{5}}
\put(250,90){\circle*{5}}
\put(220,10){\circle*{5}}
\put(220,50){\circle*{5}}
\put(220,90){\circle*{5}}
\put(250,10){\line(0,1) {80}}
\put(220,10){\line(0,1) {80}}
\put(220,10){\line(1,0) {30}}
\put(220,50){\line(1,0) {30}}
\put(220,90){\line(1,0) {30}}

\put(205,8) {$v_1$}
\put(205,48) {$v_8$}
\put(205,88) {$v_5$}

\put(256,8) {$v_6$}
\put(256,48) {$v_3$}
\put(256,88) {$v_{10}$}

\put(200,110) {Induced domino in $G_P$}
\put(25,110) {Poset $P$}
\end{picture}
\caption{A counterexample to Lemma 4.5.}
\label{f:protipriklad3}
\end{center}
\end{figure}

Let $P$ be the poset depicted in Figure~\ref{f:protipriklad3}. Clearly
that it is a counterexample to Lemma~4.5. \cite{Bres}. Indeed, $P$ does not
contain any of the posets $D_1$, $D_2$, $D_3$, $D_4$, $D_5$, $D_6$ or $D_7$ as
an \emph{isometric subposet}. But $G_P$ contains an induced domino on vertices
$v_1,v_2, v_4,v_5, v_7$, a contradiction.

{\bf\subsection{A counterexample to Theorem 4.1 from \cite{Bres3} }}

\begin{theorem*}[Theorem 4.1~\cite{Bres3}]
Let $P$ be a poset. Then $G_P$ is a cograph if and only if $P$ contains neither any of $Q_1$, $Q_2$, \ldots, $Q_7$ nor duals of $Q_2$ or $Q_5$ as an isometric subposet, see Fig.~\ref{f:posetyQ}.
\end{theorem*}

\begin{figure}[h]
\begin{center}
\includegraphics[width=0.7\textwidth]{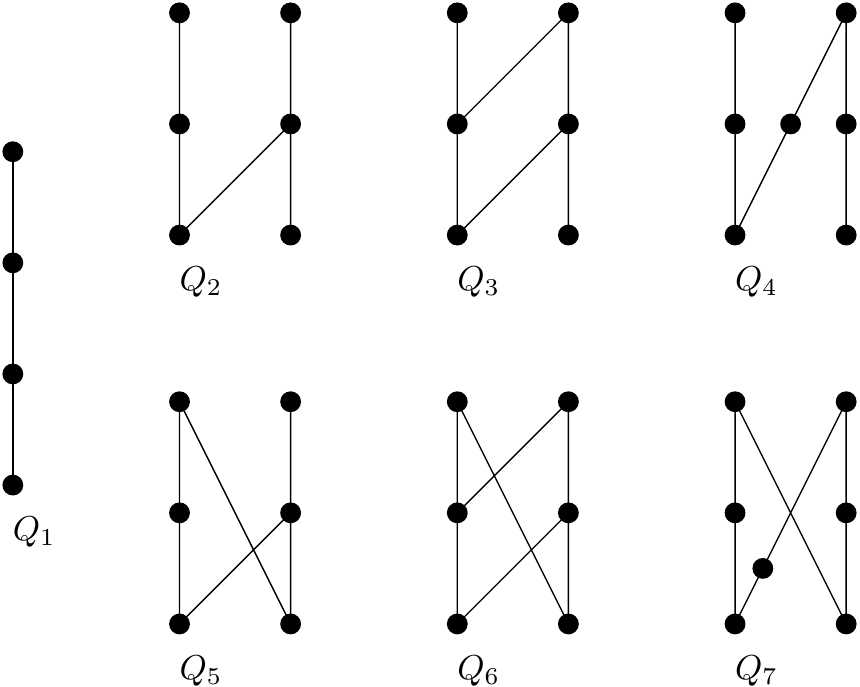}
\caption{Subposets $Q_i$, $i=1, \dots 7$.}
\label{f:posetyQ}
\end{center}
\end{figure}

\begin{figure}[h]
\begin{center}
\begin{picture}(300,100)
\put(50,10){\circle*{5}}
\put(50,30){\circle*{5}}
\put(50,50){\circle*{5}}
\put(50,70){\circle*{5}}

\put(80,40){\circle*{5}}
\put(50,10){\line(0,1) {60}}
\put(50,10){\line(1,1) {30}}
\put(50,70){\line(1,-1) {30}}
\put(35,8) {$v_1$}
\put(35,28) {$v_2$}
\put(35,48) {$v_3$}
\put(35,68) {$v_4$}

\put(85,38) {$x$}
\put(50,88) {Poset $P$}

\put(250,10){\circle*{5}}
\put(250,30){\circle*{5}}
\put(250,50){\circle*{5}}
\put(250,70){\circle*{5}}

\put(280,40){\circle*{5}}
\put(250,10){\line(0,1) {60}}
\put(250,10){\line(1,1) {30}}
\put(250,30){\line(3,1) {30}}
\put(250,50){\line(3,-1) {30}}
\put(250,70){\line(1,-1) {30}}

\put(235,8) {$v_1$}
\put(235,28) {$v_2$}
\put(235,48) {$v_3$}
\put(235,68) {$v_4$}

\put(285,38) {$x$}

\put(240,88) {C-I graph $G_P$}
\end{picture}
\caption{A counterexample to Theorem 4.1}
\label{f:protipriklad4}
\end{center}
\end{figure}
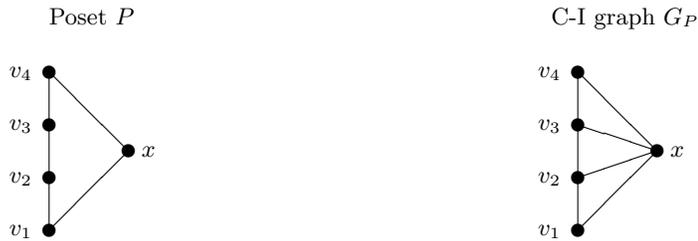

Let $P$ be the poset depicted in Fig.~\ref{f:protipriklad3}. It is easy to see
that it is a counterexample to Theorem 4.1 \cite{Bres3}. Indeed, $P$ contains
neither any of the posets $Q_1$, $Q_2$, \dots, $Q_7$ nor the duals of $Q_2$ or
$Q_5$ as an \emph{isometric subposet}. But $G_P$ contains an induced path on
four vertices $v_1,v_2, v_3,v_4$. Thus, $G_P$ is not a cograph, a
contradiction.

\section{Restatements and proofs}\label{s:oprava}

The mistake originated in Theorem~2.4~\cite{Bres}.

\begin{theorem*}[Theorem~2.4~\cite{Bres}]
Let $\cal G$ be a class of graphs with a forbidden induced subgraphs characterization. Let ${\cal P}=\{ P \  | \ P $ {\rm is a poset with} $G_P \in {\cal G} \}$. Then $\cal P$ has a forbidden isometric characterization.
\end{theorem*}

If we go carefully through the proof of this theorem in \cite{Bres} we notice
that it is not proved that the poset $P$ contains one of the constructed
posets $\{P_i\} _{i \in I}$ as an \emph{isometric} subposet. The condition of
isometry is too strong and it has to be replaced by the weaker concept of
$\prec$-preserving subposet. See Section \ref{s:terminology} for the
definition.

\begin{theorem}[see Theorem 2.4~\cite{Bres}]\label{t:opravena}
Let $\cal G$ be a class of graphs with a characterization by forbidden induced subgraphs. Let ${\cal P}=\{ P \  | \ P $ {\rm is a poset with} $G_P \in {\cal G} \}$. Then $\cal P$ has a characterization by forbidden $\prec$-preserving subposets.
\end{theorem}

For the proof of this theorem we need a slightly stronger version of Lemma~2.3~\cite{Bres}.

\begin{lemma}[see Lemma 2.3~\cite{Bres}]\label{l:opravene}
Let $Q$ be a  $\vartriangleleft$-preserving subposet of a poset $P$. Then $G_Q$ is isomorphic to a subgraph of $G_P$ induced by the points of $Q$.
\end{lemma}

\begin{proof}
Let $H$ be the subgraph of $G_P$ induced by the points of $Q$. Let $u$ and $v$ be arbitrary points in $Q$. We show that
$$ \{u,v\} \in E(H) \Leftrightarrow  \{u,v\} \in E(G_Q).$$

First suppose that $\{u,v\} \in E(H)$. This happens if and only if either $u \prec _P v$, or $v \prec _P u$, or $u \parallel _P v$. As $Q$ is a
  $\vartriangleleft$-preserving subposet of $P$ we have
\begin{alignat*}{4}
u \prec _P v &\Rightarrow\ &u \prec _Q v &\Rightarrow \{u,v\} \in E(G_Q),\\
v \prec _P u &\Rightarrow\ &v \prec _Q u &\Rightarrow \{u,v\} \in E(G_Q),\\
u \parallel _P v &\Rightarrow\ &u \parallel _Q v &\Rightarrow \{u,v\} \in E(G_Q).
\end{alignat*}

Thus if $\{u,v\} \in E(H)$ then also $\{u,v\} \in E(G_Q)$.

Now suppose that $\{u,v\} \notin E(H)$. Then $u \prec \prec_P v$ or $v \prec \prec_P u$. As $Q$ is a   $\vartriangleleft$-preserving subposet of $P$ it follows that
$u \prec \prec_Q v$ or $v \prec \prec_Q u$, and thus    $\{u,v\} \notin E(G_Q)$.

We conclude that $H$ and $G_Q$ are isomorphic graphs as stated.
\qed
\end{proof}

Now we are ready to prove Theorem~\ref{t:opravena}.

\begin{proof}[(of Theorem~\ref{t:opravena})]
Let $G_{\textrm{forb}}$ be one of the forbidden induced subgraphs for the class  $\cal G$. Let $P \in {\cal P}$ be any poset in the class ${\cal P}$. By the definition of ${\cal P}$, $G_P$ does not contain $G_{\textrm{forb}}$ as an induced subgraph. By Lemma~\ref{l:opravene}, $P$ does not contain any $\prec$-preserving subposet $Q$ such that $G_Q$ is isomorphic to $G_{\textrm{forb}}$. Hence any subposet $Q$ s.t. $G_Q$ is isomorphic to $G_{\textrm{forb}}$ is forbidden for ${\cal P}$. Repeating this for all the forbidden induced subgraphs for ${\cal G}$ we find a list of forbidden $\prec$-preserving subposets $\{ Q_i \}_{i \in I}$.

We will show that the class ${\cal P}$ is characterized by forbidden $\prec$-preserving subposets $\{ Q_i \}_{i \in I}$.

First, let $P \in {\cal P}$. Then $P$ clearly contains no $Q_i$ as a $\prec$-preserving subposet. Otherwise (by Lemma~\ref{l:opravene}) the graph $G_P$ would contain a forbidden induced subgraph for ${\cal G}$.

Conversely, suppose that $P$ contains no $Q_i$ as a $\prec$-preserving subposet. Then (by the construction of $\{ Q_i \}_{i \in I}$) $G_P$ contains no forbidden subgraph for ${\cal G}$. Thus $G_P \in {\cal G}$, and hence $P \in {\cal P}$.
\qed
\end{proof}

The previous theorem can be applied for various graph classes that admit a characterization by forbidden induced subgraphs, such as chordal graphs, claw-free graphs, distance-hereditary graphs, Ptolemaic graphs etc.

\begin{theorem}[corrected Lemma~5.1~\cite{Bres}]
\label{t:opravena claw-free}
Let $P$ be a poset. Then $G_P$ contains an induced claw if and only if $P$ contains one of $S_1$, $S_2$, $S_3$ or $S_2^*$ (the dual of $S_2$) as a $\prec$-preserving  subposet, see Fig. ~\ref{f:posetyS}.
\end{theorem}

\begin{proof}
If $P$ contains one of the posets $S_1$, $S_2$, $S_3$ or $S_2^*$ as a $\prec$-preserving  subposet then clearly $G_P$ contains an induced claw.

Conversely, suppose that $G_P$ contains an induced claw. We want to find $S_1$, $S_2$, $S_3$ or $S_2^*$ as a $\prec$-preserving  subposet of $P$.
Let us denote by $x$ the middle vertex and by $u,v,w$ the other vertices of the claw. By Lemma~\ref{l:basic}(iii), as $u,v,w$ form an independent set in $G_P$ they lie on a common chain in $P$. Without loss of generality we may suppose that $u \prec \prec \, v \prec \prec \, w$.

Note that $x \prec \, v$ is not possible, otherwise $x \prec \prec \, w$ and hence $\{ x,w \} \notin E(G_P)$, a contradiction. Similarly, it is not possible that
$x \prec \, u$, $v \prec \, x$ or $w \prec \, x$. Thus there are only five cases to distinguish:

\begin{itemize}
  \item \emph{Case 1} $x \! \parallel  \! u$, $x  \! \parallel  \!  v$, $x  \! \parallel  \!  w$. Then $P$ obviously contains $S_3$ as a $\prec$-preserving  subposet.

\item \emph{Case 2} $u \prec  \, x$, $x  \! \parallel  \!  v$, $x  \! \parallel  \!  w$. Then $P$ obviously contains $S_2$ as a $\prec$-preserving  subposet.

\item \emph{Case 3} $x \prec  \, w$, $x  \! \parallel  \!  v$, $x  \! \parallel  \!  u$. Then $P$ obviously contains $S_2^*$ as a $\prec$-preserving  subposet.

\item \emph{Case 4} $u \prec  \, x$, $x \prec  \, w$, $x  \! \parallel  \!  v$ and the length of the shortest chain in $P$ between $u$ and $w$ is equal to 4.  Then $P$ obviously contains $S_3$ as a $\prec$-preserving  subposet.

\item \emph{Case 5} $u \prec  \, x$, $x \prec  \, w$, $x  \! \parallel  \!  v$ and the length of the shortest chain in $P$ between $u$ and $w$ is greater than 4.  Then $P$ obviously contains $S_2$ as a $\prec$-preserving  subposet.
\end{itemize}
\qed
\end{proof}

Now let us restate the corresponding statements from  \cite{Bres} and
\cite{Bres3}. We skip their proofs as they are the same as the ones presented
in \cite{Bres} and \cite{Bres3}, the only mistake was claiming that the
forbidden subposets must be isometric subposets of $P$.

\begin{theorem}[corrected Lemma 4.4~\cite{Bres}]
Let $P$ be a poset. Then $G_P$ contains an induced house if and only if $P$ contains one of $R_1$, $R_2$, $R_3$, $R_4$, $R_5$ or its duals as a $\prec$-preserving subposet, see Fig.~\ref{f:posetyR}.
\end{theorem}

\begin{theorem}[corrected Lemma 4.5~\cite{Bres}]
Let $P$ be a poset. Then $G_P$ contains an induced domino if and only if $P$ contains one of $D_1$, $D_2$, $D_3$, $D_4$, $D_5$, $D_6$, $D_7$ or its duals as a $\prec$-preserving subposet, see Fig.~\ref{f:posetyD}.
\end{theorem}

Let us remark that for $P_1$, $P_2$, and $P_3$ the notion of isometric subposet and $\prec$-preserving subposet coincide. More precisely, a poset $P$ contains $P_1$, $P_2$, or $P_3$ as an isometric subposet if and only if $P$ contains $P_1$, $P_2$, or $P_3$ as a $\prec$-preserving subposet. This is because the length of the longest chain in $P_1$, $P_2$, and $P_3$ is only two. Hence, Theorem 3.1 \cite{Bres3} holds as it was stated in \cite{Bres3}.

\begin{theorem}[corrected Theorem 4.1~\cite{Bres3}]
Let $P$ be a poset. Then $G_P$ is a cograph if and only if $P$ contains none of $Q_1$, $Q_2$, \dots, $Q_7$ and neither of the duals of $Q_2$ and $Q_5$ as a $\prec$-preserving subposet, see Fig.~\ref{f:posetyQ}.
\end{theorem}

\begin{figure}
\begin{center}
\includegraphics[width=0.8\textwidth]{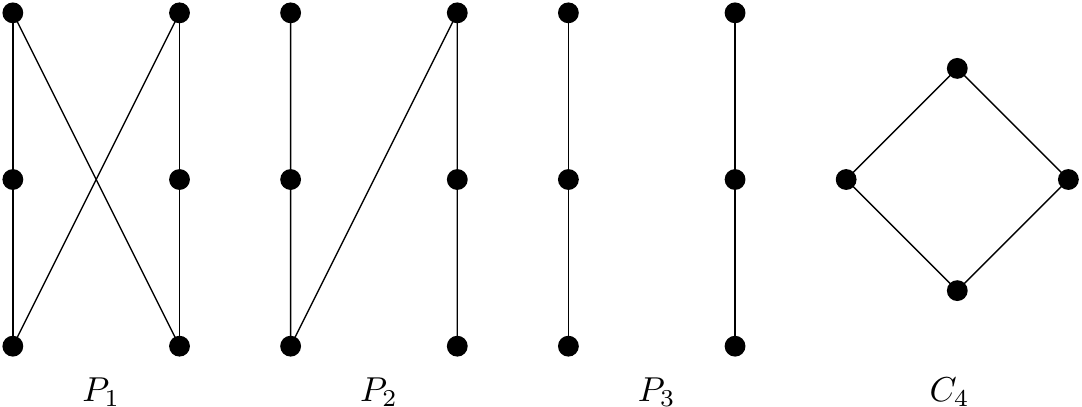}
\caption{Subposets $P_1$, $P_2$, $P_3$ and $C_4$} \label{f:3}
\label{f:posetyP}
\end{center}
\end{figure}

\section*{Acknowledgments}
Jan Bok would like to acknowledge the support by the Center of Excellence - ITI (P202/12/G061 of GA\v{C}R). Jan Bok was also
partly supported by the project GAUK 1158216 and GAUK 1334217.

The authors would like to thank Bo\v{s}tjan Bre\v{s}ar, Sandi Klav\v{z}ar, and
the editor of this paper for many valuable comments and suggestions regarding this paper.


\begin{thebibliography}{99}

\bibitem{Bres}
Bre\v{s}ar, B., Changat, M., Klav\v{z}ar, S., Kov\v{s}e, M., Mathews, J.,
Mathews, A.: Cover-Incomparability Graphs of Posets. Order {\bf
25}, 335-347 (2008)

\bibitem{Bres2}
Bre\v{s}ar, B., Changat, M., Gologranc, T., Mathews, J., Mathews, A.:
Cover-incomparability graphs and chordal graphs. Discrete Applied
Mathematics 158, 1752-1759 (2010)

\bibitem{Bres3}
Bre\v{s}ar, B., Changat, M., Gologranc, T., Sukumaran, B.:
Cographs which are cover-incomparability graphs of posets. Order 32(2), 179-187 (2015)

\bibitem{Bres4}
Bre\v{s}ar, B., Changat, M., Gologranc, T., Kov\v{s}e, M., Sukumaran, B.:
Cover-incomparability graphs and 2-colored diagrams of posets,
Taiwanese Journal of Mathematics

\bibitem{Lek}
Lekkerkerker, C.B.\ and Boland, J.C.: Representation of finite graphs by a set of intervals on the real line, Fund. Math. 51: 45-64 (1962)

\bibitem{Max}
Maxov\'{a}, J., Pavl\'{\i}kov\'{a}, P., Turz\'{\i}k, D.: On the complexity of cover-incomparability graphs of posets. Order 26, 229-236 (2009)

\bibitem{MaxDH}
Maxov\'{a}, J., Turz\'{\i}k, D.: Which distance-hereditary graphs are cover-incomparability graphs? Discret. Appl. Math. 161, 2095-2100 (2013)

\bibitem{Maxktrees}
Maxov\'{a}, J., Dubcov\'{a}, M., Pavl\'{\i}kov\'{a}, P., Turz\'{\i}k, D.: Which k-trees are cover-incomparability graphs? Discret. Appl. Math. 167, 222-227 (2014)

\end{thebibliography}
\end{document}